\documentclass[a4paper,fleqn]{cas-sc}

\usepackage[numbers]{natbib}
\usepackage{hyperref}
\usepackage{amsfonts}
\usepackage{mathrsfs}
\usepackage{amssymb}
\usepackage{amsmath}
\usepackage{amsthm}
\usepackage{enumerate}
\usepackage{color}
\usepackage{filecontents}
\usepackage{url}
\newtheorem{theorem}{Theorem}[section]
\newtheorem{remark}[theorem]{Remark}
\newtheorem{lemma}[theorem]{Lemma}

\numberwithin{equation}{section}

\usepackage[justification=centering]{caption}
\DeclareMathAlphabet{\mathcal}{OMS}{cmsy}{m}{n}
\let\mathbb\relax
\DeclareMathAlphabet{\mathbb}{U}{msb}{m}{n}
\def\tsc#1{\csdef{#1}{\textsc{\lowercase{#1}}\xspace}}
\tsc{WGM}
\tsc{QE}

\begin{document}
	\let\WriteBookmarks\relax
	\def\floatpagepagefraction{1}
	\def\textpagefraction{.001}
	\let\printorcid\relax
	\shorttitle{A Markov representation of Perron-Frobenius eigenvector for infinite non-negative matrix and Metzler-matrix}    
	%
	\shortauthors{Qian DU and Yong-Hua MAO}  
	
	\title [mode = title]{A Markov representation of Perron-Frobenius eigenvector for infinite non-negative matrix and Metzler-matrix}  
	
	%
	
	%
	
	%
	%
	%
	%
	
	
	\author{Qian DU} 
	\ead{email: duqian@mail.bnu.edu.cn}
	\author{Yong-Hua MAO}
	
	\affiliation{organization={School of Mathematical Sciences},
		addressline={Beijing Normal University,
			Laboratory of Mathematics and Complex Systems, Ministry of Education}, 
		city={	Beijing},
		postcode={100875}, 
		country={People's Republic of 	China}}
	
	
	%
	%
	%
	%
	%
	%
	
	
	\begin{abstract}
	We will represent the so-called Perron-Frobenius eigenvector (if exists) for infinite non-negative matrix $A$ and Metzler matrix by using its corresponding Markov chain with probability transition function.
	\end{abstract}
	
	
	
	
	\maketitle
	
	\section{Introduction}
	Assume that $E$ is countable, let $A=(a_{ij})_{i,j\in E}$ be a non-negative and irreducible matrix.~Assume that $\sum_{j\in E}a_{ij}=:f_i<\infty$,~$\forall~i\in E.$ 
	Define the probability transition matrix $M=(m_{ij})_{i,j\in E}$ by 
	$$
	m_{ij}=a_{ij}/f_i,~~~~~\forall~i,~j\in E.
	$$
	Then we will represent the so-called Perron-Frobenius eigenvector (if exists) for $A$  by using the Markov chain $X$ with probability transition matrix $M$.~Let $\tau_{j}^+=\inf\{n\geq1:X_n=j\}$ be the first return time to $j\in E$. Introduce some notations and definitions concerning  matrix $A$.~Let $R$ be the convergence parameter for $A$:
	$$
	R=\inf\{r\geq0:\sum_{n=0}^{\infty}r^na_{ij}^{(n)}=\infty\}=\sup\{r\geq0:\sum_{n=0}^{\infty}r^na_{ij}^{(n)}<\infty\}.
	$$
	By irreducibility, $R$ is independent of $i,~j\in E$.
	Generally speaking,~$\sum_{n=0}^{\infty}R^{n}a_{ij}^{(n)}$~can be finite or infinite.
	$A$ is called $R$-recurrent if 
	$$
	\sum_{n=0}^{\infty}R^na_{ij}^{(n)}=\infty.
	$$
	 We have the following theorem for the infinite non-negative matrix, which is a generalization of the Perron-Frobenius eigenvector in \cite{2017markov} for a finite primitive matrix.
	\begin{theorem}\label{theorem 1.7}
		Fix some $k\in E$. Assume
		
		\begin{enumerate} 
			\item [(1)] $A$ is non-negative and irreducible on $E$;
			\item [(2)] $f(i)=\sum_{j\in E}a_{ij}<\infty$ for $i\in E$;
			\item [(3)] $A$ is $R$-recurrent with $R>0$.
		\end{enumerate} 
		Then 
		$$
		u_i=\mathbb{E}_k\left(\sum\limits_{n=0}^{\tau_{k}^{+}-1}\left(I_{[X_n=i]}R ^{n}\prod\limits_{m=0}^{n-1}f(X_m)\right)\right)\in (0,\infty),~~~~i\in E
		$$
		and $u=(u_i)_{i\in E}$ satisfies
		$$
		uA =(1/R )u.
		$$
	\end{theorem}
	\begin{remark}
		We can generalize Theorem \ref{theorem 1.7} in some sense.~Suppose there exists a positive vector $\alpha=(\alpha_i)$ such that $\sum_{j}a_{ij}\alpha_{j}=M_i<\infty$. Let $\widetilde{a}_{ij}=a_{ij}\alpha_j,~\forall~i,~j\in E$. 
		Then Theorem \ref{theorem 1.7} holds for matrix $\widetilde{A}=(\widetilde a_{ij})_{i,j\in E}.$
	\end{remark}
Next we  represent $u$ for a Metzler matrix $M=(m_{ij})_{i,~j\in  E}:~m_{ij}>0,~i\not=j,~\forall~i,~j\in E$. Assume that $M$ is irreducible:~$\exists$ different $i=i_0,~\cdots,~i_m=j\in {E}$ such that $m_{ii_1}\cdots m_{i_{m-1}j}>0$. Assume $d=\sup_{i\in E}m_{ii}<\infty$. 
 Denote $d_i=\sum_{j\in E}m_{ij}$. Assume $d_i<\infty$ for $i\in E$. Let $q_{ij}=m_{ij}-d_i\delta_{ij}$. Then $Q=(q_{ij})_{ij\in E}$ is conservative and irreducible $Q$-matrix, with corresponding minimal $Q$-process $(X_t)_{t\geq 0}$. 
 Let $\lambda_d=1/R_d-d$, where $R_d$ is the convergence parameter of $M+dI$. We will prove that  $\lambda>m_{ii},~i\in E$ in Lemma \ref{exist}.
 To represent $u$, we need $\bar{M}=(\bar{m}_{ij})_{i,~j\in E}$ similar to embedded chain 
 \begin{equation}
 	\bar{m}_{ij}=\left\{
 	\begin{aligned}
 		\frac{m_{ij}}{\lambda-m_{ii}},~~~&i\not=j,\\
 		0,~~&i=j.
 	\end{aligned}
 	\right.
 \end{equation} 
  Let $\sigma_{j}^+=\inf\{t\geq\xi_1:X_t=j\}$ be the first return time to $j\in E$,~where $\xi_1=\inf\{t>0:X_t\not=X_0\}.$ Denote $q_i=-q_{ii},~i\in E$ by convention. We have the following theorem.
\begin{theorem}
	 Fix $k\in E$. Assume
	 \begin{enumerate}
	 	\item [(1)] $M=(m_{ij})$ is irreducible Metzler matrix on $E$;
	 	\item [(2)] $d_i=\sum_{j\in E}m_{ij}<\infty$ for $i\in E$;
	 	\item [(3)] $d=\sup_{i\in E}m_{ii}<\infty$.
	 	\item [(4)] $\bar{M}$ is recurrent on $E$.
	 \end{enumerate}Then
	$$u_i=\mathbb{E}_j\int_{0}^{\infty}e^{-\lambda t+\int_{0}^{t}d_{X_s}ds}I_{[X_t=j,~t<\sigma_{j}^+]}dt,~~~~i\in E$$
	and $u=(u_i)_{i\in E}$ satisfies $uM=\lambda u$.
\end{theorem}
\section{A Markov chain representation for non-negative matrix}
\noindent Let $A=(a_{ij})_{i,j\in E}$ be a non-negative matrix.~Assume $A$ is irreducible,~i.e.~$\forall~i,~j\in E$,~$\exists~n\geq1$ such that $a_{ij}^{(n)}>0$, where $A^n=(a_{ij}^{(n)}).$
Let $R$ be the convergence parameter  for $A$:
$$
R=\inf\{r\geq0:\sum_{n=0}^{\infty}r^na_{ij}^{(n)}=\infty\}
=\sup\{r\geq0:\sum_{n=0}^{\infty}r^na_{ij}^{(n)}<\infty\}.
$$
$A$ is called $R$-recurrent if 
$
\sum_{n=0}^{\infty}R^na_{ij}^{(n)}=\infty.
$
We will prove that when $A$ is $R$-recurrent there exists the unique $R$-invariant measure $x=(x_i)_{i\in E}:xA=x/R$ and $R$-invariant vector $y=(y_i)_{i\in E}: Ay=y/R.$
Define $_{r}a_{ij}^{(1)}=a_{ij},$ and inductively 
\begin{equation}\label{def}
	_{r}a_{ij}^{(n)}=\sum\limits_{l\not=r}{_{r}{a_{il}^{(n-1)}}}a_{lj}=\sum\limits_{i_1,\cdots,i_{n-1}\not=r}a_{ii_1}a_{i_1i_2}\cdots a_{i_{n-1}j},~~~~~i,~j,~r\in E.
\end{equation}

\begin{lemma}\label{key}
	Fix $k\in E$. If $R\in(0,\infty)$ such that $\sum\limits_{n=0}^{\infty}R^{n}a_{kk}^{(n)}=\infty$,  then 
	\begin{itemize}
		\item [(1)] $\sum_{n=1}^{\infty}R^n {_ka_{kk}^{(n)}}=1,$ and
		\item [(2)] 
		$
		\sum_{n=1}^{\infty}R^n{_ka_{ki}^{(n)}}\in (0,\infty)
		$
		for $i\in E$.
	\end{itemize}
\end{lemma}
\begin{proof}
	(1) It can be derived from $	a_{ij}^{(n)}=\sum_{m=1}^{n}{_{j}a_{ij}^{(m)}}a_{jj}^{(n-m)}$ for $n\geq 1,~i,~j\in E$.
	(2)  There is $m>0$ such that $_k{a_{ik}^{(m)}}>0$ by  irreducibility of $A$. We also have for all $n\geq1$,
	$$
		_k{a_{ki}^{(n)}}_k{a_{ik}^{(m)}}\leq {_k{a_{kk}^{(m+n)}}}.
	$$
	Multiplying both sides by $R^n$ and summing it over $n$, it follows from irreducibility and $R$-recurrence of $A$ that
	$$
		0<\sum\limits_{n=1}^{\infty}R^n{_k{a_{ki}^{(n)}}}\leq\frac{1}{R^m{_k{a_{ik}}^{(m)}}}	\sum_{n = 1}^{\infty}R^{m+n}{_k{a_{kk}^{(m+n)}}}<\frac{1}{R^m{_k{a_{ik}^{(m)}}}}<\infty.
	$$
\end{proof}

Now we are ready to prove Theorem \ref{theorem 1.7}.
\begin{proof}[Proof of Theorem~1.1]Notice that $\forall ~i\not=k$,
	\begin{equation}\label{finite}
		\begin{aligned}
			u_i=&\mathbb{E}_{k}\left(\sum_{n=1}^{\tau_k^{+}-1} \left(I_{[X_n=i]} R^{n}\prod_{m=0}^{n-1} f\left(X_m\right)\right)\right)\\
			&=\sum_{n= 1}^{\infty}R^{n}\mathbb{E}_{k}\left(I_{[X_n=i,X_m\not=k,1\leq m\leq n-1]}\prod_{m=0}^{n-1} f\left(X_m\right)\right) \\ 
			&=\sum_{n= 1}^{\infty}\sum_{i_1, \ldots, i_{n-1} \neq k}R^{n} f(k) f\left(i_1\right) \cdots f(i_{n-1})\\ &\quad\quad\quad\quad\quad\quad\quad\quad\quad\quad\quad\quad\quad\cdot\mathbb{P}_k[X_1=i_1,\cdots,X_{n-1}=i_{n-1},X_{n}=i]\\ 
			& =\sum_{n= 1}^{\infty}\sum_{i_1, \ldots, i_{n-1} \neq k} R^{n} \prod_{m=0}^{n-1}f(i_m)\mathbb{P}_k[X_{m+1}=i_{m+1}|X_m=i_m]\\
			&\quad\quad\quad\quad\quad\quad\quad\quad\quad\quad\quad\quad\quad\quad\quad\quad\quad\quad\quad\quad(\text{by Markov property})\\
			&=\sum_{n= 1}^{\infty}  R^{n} {_ka_{ki}^{(n)}},
		\end{aligned}
	\end{equation}
	where we set $i_0=k,~i_n=i$.~Combined with $u_k=1$, it follows from (\ref{finite}) and  Lemma \ref{key}  that $0<u_i<\infty$ for any $i\in E$. Then
	\begin{equation}\label{method}
		\begin{aligned}
			\sum_{i\in E}u_i a_{ij}
			& =\sum_{i\in E}\sum_{n = 0}^{\infty} \mathbb{E}_{k}\left(I_{[\tau_{k}^{+}>n]} R^{n}\prod_{m=0}^{n-1} f\left(X_m\right)  f(i) \mathbb{P}[X_{n+1}=j|\mathscr{F}_{n}] I_{[X_n=i]}\right) \\
			&= \sum_{i\in  E}\sum_{n = 0}^{\infty} \mathbb{E}_{k}\left(I_{[\tau_{k}^{+}>n]} R^{n}\prod_{m=0}^n f\left(X_m\right) I_{[X_n=i]} I_{[X_{n+1}=j]}\right) \\
			&=(1/R)\mathbb{E}_{k}\left(\sum_{n=1}^{\tau_{k}^{+}} \left(I_{[X_{n}=j]} R^{n}\prod_{m=0}^{n-1} f\left(X_m\right)\right)\right),
		\end{aligned}
	\end{equation}
	where $\mathscr{F}_{n}=\sigma\{X_m,m\leq n\}.$~Then we need to verify that the last expectation of (\ref{method}) is equal to $u_j$.
	For $j\not=k$,~there is no possibility $X_0=j$ or $X_{\tau_{k}^{+}}=j$.~So we have 
	\begin{equation}\label{j}
		\sum_{i\in E}u_i a_{ij}=(1/R)u_j,~~~~~j\not=k.
	\end{equation}
	\noindent For $j=k$, using the same method as (\ref{finite}), we have
	$$
	\mathbb{E}_{k}\left(\sum_{n=1}^{\tau_{k}^{+}} \left(I_{[X_{n}=k]} R^{n}\prod_{m=0}^{n-1} f\left(X_m\right)\right)\right)
	=\sum_{n= 1}^{\infty} R^n{_ka_{kk}^{(n)}}=1.~~~(\text{by Lemma~\ref{key}})
	$$
	It follows from  $u_k=1$ that
	\begin{equation}\label{k}
		\sum\limits_{i\in E}u_ia_{ik}=(1/R)\mathbb{E}_{k}\left(\sum_{n=1}^{\tau_{k}^{+}} \left(I_{[X_{n}=k]} R^{n}\prod_{m=0}^{n-1} f\left(X_m\right)\right)\right)=(1/R)u_k.
	\end{equation}
	From (\ref{j}) and (\ref{k}) we obtain
	$$
	\sum\limits_{i\in E}u_ia_{ij}=(1/R)u_j,~~~j\in E.
	$$	
	Similarly we can get 
	$$y_i=\mathbb{E}_{i}\left(\sum_{n=1}^{\tau_k^{+}} \left(I_{[X_n=k]} R^{n}\prod_{m=0}^{n-1} f\left(X_m\right)\right)\right)\in(0,\infty),~~~i\in E$$
	and $y=(y_i)_{i\in E}$ satisfies $Ay=y/R$.
\end{proof}
\noindent We note that 
\begin{equation}\label{rk}
	\sum_{i\in E}u_{i}=\mathbb{E}_k\left(\sum\limits_{n=0}^{\tau_{k}^{+}-1}\left(R ^{n}\prod\limits_{m=0}^{n-1}f(X_m)\right)\right).
\end{equation}
If $A$ is  a stochastic matrix with $R=1$, then (\ref{rk}) indicates that $A$ is positive recurrent if $\mathbb{E}_k\tau_{k}^{+}<\infty$. 
But it should be noted that in the stochastic case it is possible for the chain to be  $R$-recurrent with $R>1$.~Take the simple random walk on a line as an example.
\section{Representation for Metzler matrix}
This section is devoted to representing $u$ for Metzler matrix. Recall the definition of $\bar{M}$.  \begin{equation}
	\bar{m}_{ij}=\left\{
	\begin{aligned}
		\frac{m_{ij}}{\lambda-m_{ii}},~~~&i\not=j,\\
		0,~~&i=j.
	\end{aligned}
	\right.
\end{equation} 
\begin{lemma}\label{exist}
	Assume that $M=(m_{ij})$ is is irreducible Metzler matrix on $E$ and $d=\sup_{i\in E}m_{ii}<\infty$. Then  $\lambda=1/R_d-d>m_{ii},~i\in E$, where $R_d$ is the convergence parameter of $A_d=M+dI$.
\end{lemma}
\begin{proof}
$A_d=M+dI$ is non-negative and irreducible on $E$ following from the irreducibility of $M$.
	It follows from \cite{seneta2006non} that there always exists positive $R_d$-invariant measure $x$ such that $uA_d\leq1/R_du$, where $R_d$ is the convergence parameter of $A_d$. That is 
	$$uM\leq(1/R_d-d)u.$$ 
Obviously $u_jm_{jj}<\sum_{i\in E}u_im_{ii}\leq (1/R_d-d)u_j,~j\in E.$
\end{proof}
\begin{proof}[Proof of Theorem 1.3]
Combined with assumption (4), it follows from (\ref{def}) and \cite[Theorem 6.2]{seneta2006non} that $x_i=\sum_{n\geq 1}{_jm_{ji}^{(n)}},~i\in E$ and
 $x=(x_i)_{i\in E}$ satisfies 
 \begin{equation}\label{invariant}
 	x\bar{M}=x.
 \end{equation}
According to  decomposition of the last jump and Markov property, it follows that
$$\mathbb{E}_j[e^{\int_{0}^{t}d_{X_s}ds}I_{[X_t=i,~\sigma_j^+>t]}]=\sum_{k\not=j,~i}q_{ki}\int_{0}^{t}e^{(d_i-q_i)(t-r)}\mathbb{E}_j[e^{\int_{0}^{r}d_{X_s}ds}I_{[X_r=k,~\sigma_j^+>r]}]dr+\delta_{ij}e^{(d_j-q_j)}t.$$
Taking Laplace transform,
$\mathbb{E}_j\int_{0}^{\infty}e^{-\lambda t+\int_{0}^{t}d_{X_s}ds}I_{[X_t=i,~\sigma_j^+>t]}dt,~i\in E$ is the minimal non-negative solution of
\begin{equation}\label{solution}
	y_i=\sum_{k\not=j,~i}y_k\frac{m_{ki}}{\lambda-m_{ii}}+\frac{\delta_{ij}}{\lambda-m_{jj}}.
\end{equation}
	Set $$y_i^{(1)}=\frac{\delta_{ij}}{\lambda-m_{jj}},~y_{i}^{(n+1)}=\sum\limits_{k\not=j,~i}y_{k}^{(n)}\frac{q_{ki}}{\lambda-m_{ii}},~n\geq 1.$$
Inductively,~$y_i^{(n)}=\frac{1}{\lambda-m_{ii}}{_{j}\bar m_{ji}^{(n-1)}}$.
Using the second iteration method (cf. \cite{hou2012homogeneous}), 
$$y_i=\sum_{n=1}^{\infty}y_i^{(n)}=\frac{1}{\lambda-m_{ii}}\sum_{n=0}^{\infty}{_{j}\bar m_{ji}^{(n)}}$$
is the minimal solution of (\ref{solution}).
It follows from the uniqueness of minimal non-negative solution and recurrence of $\bar{M}$($\sum_{n\geq1}{_{j}\bar m_{jj}^{(n)}}=1$) that
\begin{equation}\label{unique}
	\mathbb{E}_j\int_{0}^{\infty}e^{-\lambda t+\int_{0}^{t}d_{X_s}ds}I_{[X_t=i,~\sigma_j^+>t]}dt=\frac{1}{\lambda-m_{ii}}\sum_{n=1}^{\infty}{_{j}\bar m_{ji}^{(n)}}.
\end{equation}
This completes the proof  follows from (\ref{invariant}) and (\ref{unique}).
\end{proof}

\bibliographystyle{cas-model2-names}
\bibliography{rerefer.bib}
\end{document}